\documentclass[12pt,leqno]{article}

\title{Classical consequences of continuous choice principles from intuitionistic analysis}
\author{Fran{\c c}ois G. Dorais}
\date{October 30, 2011\\{\small(Revised July 26, 2012)}}

\usepackage{verbatim}
\usepackage{enumerate}
\usepackage{amsmath}
\usepackage{amssymb}
\usepackage{amsthm}
\usepackage{amsrefs}

\newcounter{proc}[section]

\theoremstyle{plain}
\newtheorem{theorem}[proc]{Theorem}
\newtheorem{corollary}[proc]{Corollary}  
\newtheorem{lemma}[proc]{Lemma}

\newtheorem{proposition}[proc]{Proposition}

\theoremstyle{definition}
\newtheorem{definition}[proc]{Definition}
\newtheorem{remark}[proc]{Remark}
\newenvironment{axiom}[1]{\displaymath\tag{#1}}{\enddisplaymath}

\mathchardef\mhyphen="2D
\newcommand{\C}{\mathbb{C}}
\newcommand{\N}{\mathbb{N}}
\newcommand{\Q}{\mathbb{Q}}
\newcommand{\R}{\mathbb{R}}

\newcommand{\cat}{\mathop{{\hat{\,}}}}
\newcommand{\seq}[1]{\langle#1\rangle}
\newcommand{\set}[1]{\lbrace#1\rbrace}
\newcommand{\res}[2]{\overline{#1}#2}

\DeclareMathOperator{\dom}{dom}

\newcommand{\rec}{\mathsf{R}}
\newcommand{\pr}[2]{\seq{#1,#2}}
\newcommand{\fst}[1]{#1\pi_0}
\newcommand{\snd}[1]{#1\pi_1}
\newcommand{\shift}[1]{#1\sigma}

\newcommand{\converges}{{\downarrow}}
\newcommand{\diverges}{{\uparrow}}

\newcommand{\IFF}{\Leftrightarrow}
\newcommand{\lthen}{\mathrel{\rightarrow}}
\newcommand{\liff}{\mathrel{\leftrightarrow}}

\newcommand{\proves}{\vdash}

\newcommand{\vset}[1]{[#1]}

\newcommand{\eha}{\ensuremath{\mathsf{E}\mhyphen\mathsf{HA}}}
\newcommand{\el}{\ensuremath{\mathsf{EL}}}

\newcommand{\rca}{\ensuremath{\mathsf{RCA}}}
\newcommand{\rcao}{\ensuremath{\mathsf{RCA}_0}}
\newcommand{\wkl}{\ensuremath{\mathsf{WKL}}}
\newcommand{\lem}{\ensuremath{\mathsf{LEM}}}
\newcommand{\markov}{\ensuremath{\mathsf{M}}}
\newcommand{\qfac}{\ensuremath{\mathsf{QF}\mhyphen\mathsf{AC}^{0,0}}}
\newcommand{\ac}{\ensuremath{\mathsf{AC}}}

\newcommand{\gc}{\ensuremath{\mathsf{GC}}}
\newcommand{\lgc}{\ensuremath{\mathsf{GC}_L}}
\newcommand{\lpo}{\ensuremath{\mathsf{LPO}}}
\newcommand{\llpo}{\ensuremath{\mathsf{LLPO}}}
\newcommand{\conax}{\ensuremath{\mathsf{CON}}}
\newcommand{\recax}{\ensuremath{\mathsf{REC}}}
\newcommand{\sucax}{\ensuremath{\mathsf{SA}}}
\newcommand{\ia}{\ensuremath{\mathsf{IA}}}
\newcommand{\qfia}{\ensuremath{\mathsf{QF}\mhyphen\mathsf{IA}}}

\newcommand{\krf}{\mathrel{\mathtt{rf}}}
\newcommand{\kneg}{\mathrm{N}_K}
\newcommand{\kcon}{\Gamma_K}
\newcommand{\kmax}{\mathsf{CN}}
\newcommand{\lrf}{\mathrel{\mathtt{lrf}}}
\newcommand{\lneg}{\mathrm{N}_L}
\newcommand{\lcon}{\Gamma_L}
\newcommand{\lmax}{\mathsf{CN}_L}

\begin{document}

\maketitle

\begin{abstract}
  \noindent
  The \emph{sequential form} of a statement \[\tag{$\dagger$}\label{E:Basic}\forall\xi(B(\xi) \lthen \exists\zeta A(\xi,\zeta))\] is the statement \[\forall\xi(\forall n B(\xi_n) \lthen \exists\zeta \forall n A(\xi_n,\zeta_n)).\]
  There are many classically true statements of the form \eqref{E:Basic} whose proofs lack uniformity and therefore the corresponding sequential form is not provable in weak classical systems.
  The main culprit for this lack of uniformity is of course the law of excluded middle.
  Continuing along the lines of Hirst and Mummert~\cite{HirstMummert11}, we show that if a statement of the form \eqref{E:Basic} satisfying certain syntactic requirements is provable in some weak intuitionistic system, then the proof is necessarily sufficiently uniform that the corresponding sequential form is provable in a corresponding weak classical system.
  Our results depend on Kleene's realizability with functions and the Lifschitz variant thereof.
\end{abstract}

\section*{Introduction}

In~\cite{Brouwer27}, Brouwer introduced the continuity theorem, which states that every function on the unit interval is (uniformly) continuous.
While many other principles of intuitionistic analysis are classically valid (e.g., the fan theorem and the bar theorem), Brouwer's continuity theorem contradicts the law of excluded middle.
Indeed, were the equality of two real numbers decidable, then the characteristic function of the singleton \(\set0\) would be an example of a discontinuous function defined on the unit interval.

Still, many formal systems of constructive analysis either satisfy Brouwer's continuity theorem, or are compatible with it.
In fact, variants of the continuity theorem are often combined with the (classically valid) choice principles to yield continuous choice principles of the form:
\begin{quote}
  \textit{If for every \(\xi\) there is a \(\zeta\) such that \(A(\xi,\zeta),\) then there is a continuous function \(F\) such that \(A(\xi,F(\xi))\) holds for all \(\xi.\)}
\end{quote}
When \(\xi\) and \(\zeta\) are interpreted as varying over the the unit interval (or the real numbers, or Cantor space, or Baire space), this enforces a highly constructive strength to the existential quantifier.
Indeed, the continuity of \(F\) allows to effectively translate finitary information about the argument \(\xi\) into finitary information about a witness \(\zeta\) to the statement \(A(\xi,\zeta).\)
Thus, even in very weak systems where infinitary constructions are hardly formalizable, one can still use \(F\) to simultaneously transform an infinite sequence of arguments \(\seq{\xi_0,\xi_1,\dots}\) into a corresponding sequence of witnesses \(\seq{\zeta_0,\zeta_1,\dots}\) such that \(A(\xi_n,\zeta_n)\) holds for every \(n.\)

This general idea was exploited by Hirst and Mummert~\cite{HirstMummert11} to show that if \(A(\xi,\zeta)\) has a special syntactic form, then \[\eha^\omega + \ac \proves \forall\xi\exists\zeta A(\xi,\zeta)\] implies \[\rca^\omega \proves \forall\xi\exists\zeta\forall n A(\xi_n,\zeta_n),\] where \(\eha^\omega\) is a system Heyting arithmetic with extensional higher types that is used in proof theory (cf.\ \cite{Kohlenbach08}), \ac\ is the full axiom of choice, and \(\rca^\omega\) is a variant with higher types of Friedman's classical system of recursive comprehension that is used in reverse mathematics (cf.\ \cite{Kohlenbach05}).

The results of Hirst and Mummert are based on Kreisel's modified realizability and G{\"o}del's dialectica interpretation.
In this paper, we use Kleene's realizability with functions and a Lifshitz variant thereof due to van Oosten to obtain similar results.
Our first result (Corollary~\ref{C:KUniformization}) shows in particular that if \(A(\xi,\zeta)\) satisfies certain syntactic requirements, then \[\el + \gc \proves \forall\xi\exists\zeta A(\xi,\zeta)\] then \[\rca \proves \forall\xi\exists\zeta\forall nA(\xi_n,\zeta_n),\] where \el\ is a system of intuitionistic analysis described in the next section and \gc\ is a strong continuous choice principle that implies Brouwer's continuity theorem.
Our second result (Corollary~\ref{C:LUniformization}) is similar except that it incorporates the weak K{\"o}nig lemma (\wkl).%
\footnote{Note that in the reverse mathematics literature, \wkl\ is normally used as an abbreviation for \rca\ together with the weak K{\"o}nig lemma.
We will avoid this practice since we also want to consider the weak K{\"o}nig lemma in intuitionistic systems.}
If \(A(\xi,\zeta)\) satisfies certain syntactic requirements, then \[\el + \wkl + \lgc \proves \forall\xi\exists\zeta A(\xi,\zeta)\] implies \[\rca + \wkl \proves \forall\xi\exists\zeta\forall n A(\xi_n,\zeta_n),\] where \lgc\ is a weakening of \gc\ that does not imply continuous choice but still implies Brouwer's continuity theorem.
This result is very interesting since \wkl\ is not generally recognized as a constructive principle.

\section{The systems \el\ and \rca}

Our base system for intuitionistic analysis is a minor variant of the system \el\ described by Troelstra~\cite[\S1.9.10]{Troelstra73}.
This is a system with two sorts: numbers and (unary) functions.
We will generally use Roman letters \(a,b,c,\ldots\) to range over number terms and Greek letters \(\alpha,\beta,\gamma,\ldots\) to range over function terms.
The terms of the language are built as follows:
\begin{itemize}
\item number variables are number terms;
\item function variables are function terms;
\item the zero constant \(0\) is a number term;
\item the successor constant \(\shift{}\) is a function term;
\item if \(t_1,\dots,t_k\) are number terms and \(f\) is a symbol for a \(k\)-ary primitive recursive function then \(f(t_1,\dots,t_k)\) is a number term;
\item if \(t\) is a number term and \(\tau\) is a function term then the evaluation \(\tau(t)\) is a number term;
\item if \(t\) is a number term and \(x\) is a number variable then \(\lambda x.t\) is a function term;
\item if \(t\) is a number term and \(\tau\) is a function term then \(\rec t \tau\) is a function term.
\end{itemize}
The only atomic relation in our language is equality for the number sort; equality for the function sort is defined by extensionality: \[\alpha = \beta \liff \forall x(\alpha(x) = \beta(x)).\]
Formulas are built in the usual way for intuitionistic systems, except that we think of the disjunction \(A \lor B\) as an abbreviation for \[\exists x ((x = 0 \lthen A) \land (x \neq 0 \lthen B)).\]
Since equality for the number sort is decidable, this is equivalent to the usual intuitionistic disjunction~\cite[\S1.3.7]{Troelstra73}.

In addition to the usual intuitionistic logic axioms, our base systems have the usual equality axioms and the defining axioms for all primitive recursive functions.
Of course, for this to make sense, the zero and successor constants must satisfy
\begin{axiom}{\sucax} 
  \sigma(x) \neq 0 \land (\sigma(x) = \sigma(y) \lthen x = y)
\end{axiom}
and the induction scheme
\begin{axiom}{\ia}
  \forall x(A(x) \lthen A(\sigma(x))) \lthen \forall x(A(0) \lthen A(x)),
\end{axiom}
where \(A(x)\) is any formula.
The last two term formation rules are governed by the \(\lambda\)-conversion scheme
\begin{axiom}{\conax} 
  (\lambda x.t)(t') = t[x/t']
\end{axiom}
and the recursion scheme
\begin{axiom}{\recax} 
  (\rec t \tau)(0) = t \land (\rec t \tau)(\shift{}(t')) = \tau((\rec t \tau)(t')).
\end{axiom}
Moreover, we have the following choice scheme
\begin{axiom}{\qfac} 
  \forall x \exists y A(x,y) \lthen \exists \alpha \forall x A(x,\alpha(x))
\end{axiom}
where \(A(x,y)\) is a quantifier-free formula.
The system \(\el_0\) is defined in exactly the same way, except that \(\ia\) is replaced by the quantifier-free induction axiom \(\qfia.\)

Although not part of our base systems, we will often make use of the \emph{Markov principle}
\begin{axiom}{\markov}
  \lnot\lnot\exists x(\alpha(x) = 0) \lthen \exists x (\alpha(x) = 0).
\end{axiom}
This principle is a simple consequence of the \emph{law of excluded middle} (\lem), which distinguishes classical systems from intuitionistic systems.
We define \(\rca\) and \(\rca_0\) to be the classical systems \(\el + \lem\) and \(\el_0 + \lem,\) respectively.
These are function-based systems which are equivalent to the set-based system of recursive comprehension (with full induction and just \(\Sigma^0_1\)-induction, respectively) traditionally used in reverse mathematics~\cite{Simpson09}.

Since our basic systems have symbols for all primitive recursive functions, pairs and sequences of numbers can be encoded in the usual manner.
The length of a finite sequence \(x\) is denoted \(|x|.\)
We write \(\seq{x_0,\ldots,x_{n-1}}\) for the finite sequence of length \(n\) whose \((i+1)\)-th term is \(x_i.\)
The concatenation of \(x\) and \(y\) is denoted \(x\cat y.\)
We will often view functions as infinite sequences of numbers.
If \(\alpha\) is a function, we write \(\res{\alpha}n\) for the finite initial segment \(\seq{\alpha(0),\ldots,\alpha(n-1)}.\)

For pairs and sequences of functions, we use the following encoding schemes.
Define \[\fst{} = \lambda n.2n, \quad \snd{} = \lambda n.2n+1.\]
If \(\alpha, \beta\) are two functions then \(\pr{\alpha}{\beta}\) denotes the unique function such that \(\fst{\pr{\alpha}{\beta}} = \alpha\) and \(\snd{\pr{\alpha}{\beta}} = \beta.\)
In a similar fashion, any function \(\alpha\) can also be viewed as an infinite sequence of functions where the \((m+1)\)-th such function is \[\alpha_m = \lambda n.\alpha(2^m(2n+1)-1).\]
When it makes sense, we will write \(\seq{\alpha_m}_{m=0}^\infty\) for the unique \(\alpha\) whose \((m+1)\)-th component is \(\alpha_m.\)
Number-function pairs are encoded by concatenation, that is \(\seq{n}\cat\alpha\) denotes the unique function such that \((\seq{n}\cat\alpha)(0) = n\) and \(\shift{(\seq{n}\cat\alpha)} = \alpha.\)

\subsection{Kleene's second algebra in \el}

Our results of Section~\ref{S:Kleene} depend on Kleene's realizability with functions.
The base system \el\ is tailored to formalize this notion of realizability.
To do this, we need to discuss the representation of partial continuous maps inside \el.

A function \(\alpha\) encodes a partial continuous map from functions to numbers defined by \[\alpha(\beta) = \alpha(\res{\beta}{n}) - 1\] where \(n\) is the unique number such that \[\alpha(\res{\beta}{n}) \neq 0 \land \forall m < n(\alpha(\res{\beta}{m}) = 0);\] if there is no such \(n,\) then \(\alpha(\beta)\) is undefined. 
We write \(\alpha(\beta)\converges\) when \(\alpha(\beta)\) is defined and we write \(\alpha(\beta)\diverges\) when \(\alpha(\beta)\) is undefined.

Similarly, \(\alpha\) encodes a partial continuous map from functions to functions defined by \[\alpha|\beta = \lambda n.\alpha(\seq{n}\cat\beta)\] provided that \(\alpha(\seq{n}\cat\beta)\converges\) for every \(n.\)
We write \(\alpha|\beta\converges\) when \(\alpha|\beta\) is defined and we write \(\alpha|\beta\diverges\) when \(\alpha|\beta\) is undefined.
We use the left associative convention for \(|,\) that is we will write \(\alpha|\beta|\gamma\) for \((\alpha|\beta)|\gamma.\)
Consequently, \(\alpha|\beta|\gamma\converges\) abbreviates \(\alpha|\beta\converges \land (\alpha|\beta)|\gamma\converges,\) and so on.

Every partial continuous map \(F\) from functions to functions whose domain is a \(G_\delta\) set admits a representation of the form \(F(\xi) = \phi|\xi.\)
We will write \(\Lambda\xi.F(\xi)\) for a function \(\phi\) that represents \(F\) in this way.
There are always multiple choices for \(\phi,\) but in all instances of this fact that we will use there is a natural choice of \(\phi\) that can be read from the description of \(F.\)

\subsection{Compact sets of functions in \el}

Our results of Section~\ref{S:Lifschitz} depend on the Lifschitz variant of realizability with functions due to van Oosten.
To formalize this notion of realizability, we need to introduce an encoding of compact sets of functions.

Every function \(\alpha\) encodes a compact set of functions defined by \[\vset{\alpha} = \set{\xi \leq \fst{\alpha} : \snd{\alpha}(\xi)\diverges}.\]
Formally, we think of \(\xi \in \vset{\alpha}\) as an abbreviation for the statement \[\forall n(\xi(n) \leq \fst{\alpha}(n) \land \snd{\alpha}(\res{\xi}{n}) = 0).\]
We will write \(\vset{\alpha} \neq \varnothing\) to assert that \(\vset{\alpha}\) is inhabited: \(\exists\xi(\xi \in \vset{\alpha}).\)

For a sound theory of compact sets, we will make frequent use of the \emph{weak K{\"o}nig lemma}:
\begin{axiom}{\wkl}
  T(\alpha,\beta) \land \forall n \exists x (|x| = n \land \alpha(x) = 0) \lthen \exists \xi \leq \beta \forall n (\alpha(\res{\xi}{n}) = 0),
\end{axiom}
where \(T(\alpha,\beta)\) says that \(\set{x : \alpha(x) = 0}\) is a tree bounded by \(\beta\): \[\forall x,y (\alpha(x\cat\seq{y}) = 0 \lthen \alpha(x) = 0 \land y \leq \beta(|x|)).\]
With this axiom, the statement \(\vset{\alpha} \neq \varnothing\) is equivalent to a \(\Pi^0_1\) formula.

Van Oosten~\cite{VanOosten90} shows that many properties of compact sets can be formalized in the theory \el\ + \wkl\ + \markov.
In particular, the following fact \cite[Lemma~5.7]{VanOosten90} will be useful.

\begin{lemma}\label{L:Image}
  There is a function term \(\iota\) such that \(\el + \wkl + \markov\) proves that \[\forall\xi \in \vset{\alpha}(\phi|\xi\converges) \lthen \iota|\pr{\phi}{\alpha}\converges \land \forall\zeta(\zeta \in \vset{\iota|\pr{\phi}{\alpha}} \liff \exists\xi \in \vset{\alpha}(\zeta = \phi|\xi)).\]
  In other words, if \(\vset{\alpha} \subseteq \dom\phi\) then \(\vset{\iota|\pr{\phi}{\alpha}} = \set{\phi|\xi : \xi \in \vset{\alpha}}.\)
\end{lemma}

\noindent
It is unclear whether the Markov principle \markov\ is necessary to establish this and other lemmas from~\cite{VanOosten90}.


\section{Classical consequences of \gc}\label{S:Kleene}

Troelstra's \emph{generalized continuity principle} is the scheme
\begin{equation*}\tag{\gc}
  \forall \xi (B(\xi) \lthen \exists \zeta A(\xi,\zeta)) \lthen 
  \exists \alpha \forall \xi (B(\xi) \lthen \alpha|\xi\converges \land A(\xi,\alpha|\xi))
\end{equation*} 
where \(B(\xi)\) is in \(\kneg\) (defined below) and \(A(\xi,\zeta)\) is arbitrary.
One immediate consequence of \gc\ is that if \(A(\xi,\zeta)\) defines the graph of a total function, then this function must be continuous.
It follows that \gc\ is plainly false in the classical system \rca.

However, we will momentarily define two classes of formulas \(\kneg\) and \(\kcon\) such that consequences of \gc\ of the form \[\forall\xi(B(\xi) \lthen \exists\zeta A(\xi,\zeta))\] where \(B(\xi)\) is in \(\kneg\) and \(A(\xi,\zeta)\) is in \(\kcon\) are not only consequences of \rca, but the sequential form \[\forall\xi(\forall n B(\xi_n) \lthen \exists\zeta \forall n A(\xi_n,\zeta_n))\] is also a consequence of \rca.

The proof of this fact relies on Kleene's realizability with functions~\cite{KleeneVesley65}, which is defined as follows.

\begin{definition}\label{D:KRF}\mbox{}
  \begin{itemize}
  \item \(\alpha \krf A\) is \(A\) for atomic \(A.\)
  \item \(\alpha \krf (A \land B)\) is \(\fst{\alpha} \krf A \land \snd{\alpha} \krf B.\)
  \item \(\alpha \krf (A \lthen B)\) is \(\forall \xi(\beta \krf A \lthen \alpha|\xi\converges \land \alpha|\xi \krf B).\)
  \item \(\alpha \krf \forall x A\) is \(\forall x(\alpha_x \krf A).\)
  \item \(\alpha \krf \forall \xi A\) is \(\forall \xi(\alpha|\xi\converges \land \alpha|\xi \krf A).\)
  \item \(\alpha \krf \exists x A\) is \(\shift{\alpha} \krf A[x/\alpha(0)].\)
  \item \(\alpha \krf \exists \xi A\) is \(\snd{\alpha} \krf A[\xi/\fst{\alpha}].\)
  \end{itemize}
\end{definition}

\noindent
Note that \(\alpha \krf A\) never involves existential quantifiers, except to say that \(\alpha|\xi\converges\) in which case the scope of the existential quantifier is quantifier-free.
It follows that \(\alpha \krf A\) always belongs to the class \(\kneg.\)%
\footnote{Elements of \(\kneg\) are called `almost negative formulas' by Troelstra~\cite{Troelstra73}.}

\begin{definition}\label{D:KNeg}\mbox{}
  \begin{itemize}
  \item If \(A\) is quantifier-free then \(A,\) \(\exists x A,\) \(\exists \xi A\) are in \(\kneg.\)
  \item If \(A, B\) are in \(\kneg\) then so are \(A \land B,\) \(A \lthen B,\) \(\forall x A,\) \(\forall \xi A.\)
  \end{itemize}
\end{definition}

In fact, the formulas of \(\kneg\) are precisely the formulas which realize themselves in the following sense~\cite[Lemma~3.3.8]{Troelstra73}.

\begin{lemma}\label{L:KNegative}
  If \(B(\xi) \in \kneg\) then \[\el \proves \exists\alpha(\alpha \krf B(\xi)) \liff B(\xi).\]
  In fact, there is a function term \(\omega_B\) such that \[\el \proves B(\xi) \liff \omega_B|\xi\converges \land \omega_B|\xi \krf B(\xi).\]
\end{lemma}

\noindent
Stated in full generality, \(B\) could depend on more than one argument (hence so would \(\omega_B\)).
However, this more general statement can be derived from Lemma~\ref{L:KNegative} by packing all the arguments into one.

Kleene's realizability with functions was given the following characterization by Troelstra~\cite[Theorem~3.3.11]{Troelstra73}.

\begin{theorem}[Characterization of \(\krf\)]\label{T:KCharacterization}
  For every formula \(A\): 
  \begin{enumerate}[\upshape\bfseries(a)]
  \item \(\el + \gc \proves A \liff \exists\alpha(\alpha \krf A)\)
  \item \(\el + \gc \proves A \IFF \el \proves \exists\alpha(\alpha \krf A)\)
  \end{enumerate}
\end{theorem}

\noindent
If \(A\) has the property that \(\el \proves \exists\alpha(\alpha \krf A) \lthen A,\) we then have \[\el + \gc \proves A \IFF \el \proves A.\]
Thus \(\el + \gc\) is conservative over \(\el\) for formulas with this property.
Lemma~\ref{L:KNegative} shows that every formula in \(\kneg\) has this property, but so do many other formulas.
 
\begin{definition}\label{D:KCon}\mbox{}
 \begin{itemize}
  \item Quantifier-free formulas are in \(\kcon.\)
  \item If \(A, B\) are in \(\kcon\) then so are \(A \land B,\) \(\forall x A,\) \(\forall \xi A,\) \(\exists x A,\) \(\exists \xi A.\)
  \item If \(A\) is in \(\kneg\) and \(B\) is in \(\kcon,\) then \(A \lthen B\) is in \(\kcon.\)
  \end{itemize}
\end{definition}

\noindent
The following fact is implicit in~\cite[Theorem~3.6.18]{Troelstra73}.

\begin{lemma}\label{L:KConservation}
  If \(A \in \kcon\) then \(\el \proves \exists \alpha (\alpha \krf A) \lthen A.\)
\end{lemma}

\noindent
Thus, by the characterization of \(\krf,\) it follows that \gc\ is conservative over \el\ for formulas in \(\kcon.\)

Together, the above results imply the following.

\begin{proposition}\label{P:KExtract}
  Suppose \(B(\xi) \in \kneg\) and \(A(\xi,\zeta) \in \kcon.\)
  If \[\el + \gc \proves \forall \xi(B(\xi) \lthen \exists \zeta A(\xi,\zeta))\] then \[\el \proves \exists\alpha\forall\xi(B(\xi) \lthen \alpha|\xi\converges \land A(\xi,\alpha|\xi)).\]
\end{proposition}

\begin{proof}
  By Theorem~\ref{T:KCharacterization}, we know that \[\el \proves \exists\beta(\beta \krf \forall \xi (B(\xi) \lthen \exists \zeta A(\xi,\zeta))).\]
  Work in \el\ and assume \(\beta \krf \forall \xi(B(\xi) \lthen \exists\zeta A(\xi,\zeta)).\)
  Unpacking Definition~\ref{D:KRF}, we see that \[\gamma \krf B(\xi) \lthen \beta|\xi|\gamma\converges \land \snd{(\beta|\xi|\gamma)} \krf A(\xi,\fst{(\beta|\xi|\gamma)}).\]
  Since \(B(\xi) \in \kneg,\) it follows from Lemma~\ref{L:KNegative} that there is a term \(\omega_B\) such that \(B(\xi) \liff \omega_B|\xi\converges \land \omega_B|\xi \krf B(\xi).\)
  Finally, since \(A(\xi,\zeta) \in \kcon,\) it follows from Lemma~\ref{L:KConservation} that \(\alpha = \Lambda\xi.\fst{(\beta|\xi|(\omega_B|\xi))}\) is as required.
\end{proof}

\noindent
By the deduction theorem, the above result also holds when \el\ is replaced by \el\ + \(\Delta,\) where \(\Delta\) is any collection of sentences from \(\kneg.\)

\begin{definition}
  Let \(\kmax\) be the set of all sentences \(A\) from \(\kneg\) such that \(\rca \proves A.\)
  In other words, \(\kmax\) consists of all consequences of the law of excluded middle which belong to the syntactic class \(\kneg.\)
\end{definition}

\noindent
Note that \(\kmax\) includes the Markov principle \markov.

Our uniformization result for this section is the following.

\begin{corollary}\label{C:KUniformization}
  Suppose \(B(\xi)\) is from \(\kneg\) and \(A(\xi,\zeta)\) is from \(\kcon.\)
  If \[\el + \gc + \kmax \proves \forall\xi(B(\xi) \lthen \exists\zeta A(\xi,\zeta))\] then \[\rca \proves \forall\xi(\forall n B(\xi_n) \lthen \exists\zeta \forall n A(\xi_n,\zeta_n)).\]
\end{corollary}

\begin{proof}
  Suppose that \[\el + \gc + \kmax \proves \forall\xi(B(\xi) \lthen \exists\zeta A(\xi,\zeta)).\]
  By Lemma~\ref{P:KExtract}, we know that \[\el + \kmax \proves \exists\alpha\forall\xi(B(\xi) \lthen \alpha|\xi\converges \land A(\xi,\alpha|\xi)).\]
  Now work in \(\rca,\) which extends \(\el+\kmax.\)
  Given \(\alpha\) such that \[\forall\xi(B(\xi) \lthen \alpha|\xi\converges \land A(\xi,\alpha|\xi)),\] if \(\xi\) is such that \(\forall n B(\xi_n),\) then \(\zeta = \seq{\alpha|\xi_n}_{n=0}^\infty\) is such that \(\forall n A(\xi_n,\zeta_n).\)
  It follows that \[\rca \proves \forall\xi(\forall n B(\xi_n) \lthen \exists\zeta \forall n A(\xi_n,\zeta_n)).\qedhere\]
\end{proof}

\noindent
Note that this proof gives much more than the conclusion of the theorem requires.
Indeed, Proposition~\ref{P:KExtract} is a much stronger result than Corollary~\ref{C:KUniformization}.
Nevertheless, Corollary~\ref{C:KUniformization} has several uses and its proof constiutes a nice warm-up for the next section.

\begin{remark}\label{R:KInduction}
  In reverse mathematics, it is traditional to use the base system \(\rca_0,\) which only postulates \(\Sigma^0_1\)-induction, rather than the system \(\rca,\) which postulates full induction.
  Unfortuately, following the proof theoretic tradition, Troelstra assumes full induction throughout~\cite{Troelstra73}.
  However, a close inspection of Troelstra's arguments shows that this assumption is not necessary to establish the characterization and conservation results for \(\krf.\)
  Therefore, Proposition~\ref{P:KExtract} and Corollary~\ref{C:KUniformization} have analogues with \(\el\) and \(\rca\) replaced by \(\el_0\) and \(\rca_0,\) respectively.
\end{remark}

\section{Classical consequences of \lgc}\label{S:Lifschitz}

Van Oosten's \emph{Lifschitz generalized continuity principle} is the scheme
\begin{multline*}\tag{\lgc}
  \forall \xi (B(\xi) \lthen \exists \zeta A(\xi,\zeta)) \lthen \\
  \exists \alpha \forall \xi (B(\xi) \lthen \alpha|\xi\converges \land \vset{\alpha|\xi} \neq \varnothing \land \forall\zeta \in \vset{\alpha|\xi} A(\xi,\zeta))
\end{multline*}
where \(B(\xi)\) is in \(\lneg\) (defined below) and \(A(\xi,\zeta)\) is arbitrary.
Unlike \gc, which offers a single witness for \(\exists\zeta A(\xi,\zeta),\) \lgc\ offers a nonempty compact set of witnesses for \(\exists\zeta A(\xi,\zeta).\)
The parameter for this compact set varies continuously with \(\xi,\) but there is no general way to continuously select a single element from this compact set.
Thus, \gc\ implies \lgc\ but the converse is false.

Nevertheless, \lgc\ still implies Brouwer's continuity theorem.
Indeed, if \(A(\xi,\zeta)\) describes the graph of a total function, then the compact set of witnesses produced by \lgc\ must be a singleton set.
Since it is possible to continuously extract the unique element of a singleton set from its parameter \cite[Lemma~5.3]{VanOosten90}, this shows that \(A(\xi,\zeta)\) describes the graph of a continuous function.
Like \gc, it follows that \lgc\ is also classically false.

Similar to the case of \gc, we will define two classes of formulas \(\lneg\) and \(\lcon\) such that consequences of \(\el + \wkl + \markov + \lgc\) of the form \[\forall\xi(B(\xi) \lthen \exists\zeta A(\xi,\zeta))\] where \(B(\xi)\) is in \(\lneg\) and \(A(\xi,\zeta)\) is in \(\lcon\) are not only consequences of \(\rca + \wkl,\) but the sequential form \[\forall\xi(\forall n B(\xi_n) \lthen \exists\zeta \forall n A(\xi_n,\zeta_n))\] is a also consequence of \(\rca + \wkl.\)
The proof of this fact relies on Lifschitz realizability with functions which was introduced by van Oosten~\cite{VanOosten90}.

\begin{definition}\label{D:LRF}\mbox{}
  \begin{itemize}
  \item \(\alpha \lrf A\) is \(A\) for atomic \(A.\)
  \item \(\alpha \lrf (A \land B)\) is \(\fst{\alpha} \lrf A \land \snd{\alpha} \lrf B.\)
  \item \(\alpha \lrf (A \lthen B)\) is \(\forall \beta(\beta \lrf A \lthen \alpha|\beta\converges \land \alpha|\beta \lrf B).\)
  \item \(\alpha \lrf \forall x A\) is \(\forall x(\alpha_x \lrf A).\)
  \item \(\alpha \lrf \forall \xi A\) is \(\forall \xi(\alpha|\xi\converges \land \alpha|\xi \lrf A).\)
  \item \(\alpha \lrf \exists x A\) is \(\vset{\alpha} \neq \varnothing \land \forall \beta \in \vset{\alpha}(\shift{\beta} \lrf A[x/\beta(0)]).\)
  \item \(\alpha \lrf \exists \xi A\) is \(\vset{\alpha} \neq \varnothing \land \forall \beta \in \vset{\alpha}(\snd{\beta} \lrf A[\xi/\fst{\beta}]).\)
  \end{itemize}
\end{definition}

\noindent
The analogue of the class \(\kneg\) is the broader class \(\lneg.\)%
\footnote{Elements of \(\lneg\) are called `\(\mathrm{B}\Sigma^1_2\)-negative formulas' by van Oosten.}

\begin{definition}\label{D:LNeg}\mbox{}
  \begin{itemize}
  \item If \(A\) is quantifier-free then \(A,\) \(\exists x A,\) \(\exists \xi A\) are in \(\lneg.\)
  \item If \(A\) is quantifier-free and \(\tau\) is a function term in which \(\xi\) does not occur then \(\exists \xi \leq \tau \forall z A\) is in \(\lneg.\)
    Similarly, if \(A\) is quantifier-free and \(t\) is a number term in which \(x\) does not occur then \(\exists x \leq t \forall z A\) is in \(\lneg.\)
  \item If \(A, B\) are in \(\lneg\) then so are \(A \land B,\) \(A \lthen B,\) \(\forall x A,\) \(\forall \xi A.\)
  \end{itemize}
\end{definition}

\noindent
With the aid of the second clause, the disjunction of one or more \(\Pi^0_1\) statements can be formulated in \(\lneg.\)
Thus, statements like the dichotomy law for Cauchy real numbers (discussed in Section~\ref{S:Applications}) can be expressed in \(\lneg\) but not in \(\kneg.\)

Again, the formula \(\alpha \lrf A\) is always in \(\lneg.\)
In fact, the formulas of \(\lneg\) are precisely the formulas which realize themselves in the following sense~\cite[Lemma~5.12]{VanOosten90}.

\begin{lemma}\label{L:LNegative}
  If \(B(\xi) \in \lneg\) then \[\el + \wkl + \markov \proves \exists\alpha(\alpha \lrf B(\xi)) \liff B(\xi).\]
  In fact, there is a function term \(\omega_B\) such that \[\el + \wkl + \markov \proves B(\xi) \liff \omega_B|\xi\converges \land \omega_B|\xi \lrf B(\xi).\]
\end{lemma}

\noindent
Again, there is a more general form of this which allows \(B\) to have more than one parameter, but this can be derived from the above by packing all arguments into one.

Lifschitz realizability with functions was characterized by van Oosten~\cite[Theorem~5.15]{VanOosten90}.%
\footnote{Note that the statement of Theorem~5.15(ii) in \cite{VanOosten90} has a typo which is corrected in our statement of Theorem~\ref{T:LCharacterization}(b).}

\begin{theorem}[Characterization of \(\lrf\)]\label{T:LCharacterization}
  For every formula \(A\): 
  \begin{enumerate}[\upshape\bfseries(a)]
  \item \(\el + \wkl + \markov + \lgc \proves A \liff \exists\alpha(\alpha \lrf A)\)
  \item \(\el + \wkl + \markov + \lgc \proves A \IFF \el + \wkl + \markov \proves \exists\alpha(\alpha \lrf A)\)
  \end{enumerate}
\end{theorem}

\noindent
Again, it is unclear whether \markov\ is necessary for this characterization of \({\lrf}.\)

The class \(\lcon\) is defined as follows.
  
\begin{definition}\label{D:LCon}\mbox{}
 \begin{itemize}
  \item Quantifier-free formulas are in \(\lcon.\)
  \item If \(A, B\) are in \(\lcon\) then so are \(A \land B,\) \(\forall x A,\) \(\forall \xi A,\) \(\exists x A,\) and \(\exists \xi A.\)
  \item If \(A\) is in \(\lneg\) and \(B\) is in \(\lcon,\) then \(A \lthen B\) is in \(\lcon.\)
  \end{itemize}
\end{definition}

\noindent
Together with the characterization of \(\lrf,\) the following fact shows that \lgc\ is conservative over \(\el + \wkl + \markov\) for formulas in \(\lcon.\)

\begin{lemma}\label{L:LConservation}
  If \(A \in \lcon\) then \(\el + \wkl + \markov \proves \exists \alpha (\alpha \lrf A) \lthen A.\)
\end{lemma}

\begin{proof}[Proof sketch.]
  The proof of this lemma is a straightforward induction on the complexity of \(A.\)
  We only prove the implication case.

  Work in \(\el + \wkl + \markov.\)
  Suppose \(\alpha \lrf (B \lthen A),\) where \(A\) is from \(\lcon\) and \(B\) is from \(\lneg.\)
  We need to show that \(B \lthen A.\)
  By definition of \(\lrf,\) we then have that if \(\beta \lrf B\) then \(\alpha|\beta\converges\) and \(\alpha|\beta \lrf A.\)
  Assume \(B.\)
  By Lemma~\ref{L:LNegative}, there is a function term \(\omega\) such that \(\omega \lrf B.\)
  It follows that \(\alpha|\omega\converges\) and \(\alpha|\omega \lrf A.\) 
  Therefore \(A,\) by the induction hypothesis.
\end{proof}

Together, the above results imply the following.

\begin{proposition}\label{P:LExtract}
  Suppose \(B(\xi) \in \lneg\) and \(A(\xi,\zeta) \in \lcon.\)
  If \[\el + \wkl + \markov + \lgc \proves \forall \xi(B(\xi) \lthen \exists \zeta A(\xi,\zeta))\] then \[\el + \wkl + \markov \proves \exists\alpha\forall\xi(B(\xi) \lthen \alpha|\xi\converges \land \vset{\alpha|\xi} \neq \varnothing \land \forall \zeta \in \vset{\alpha|\xi} A(\xi,\zeta)).\]
\end{proposition}

\begin{proof}
  By Theorem~\ref{T:LCharacterization}, we know that \[\el + \wkl + \markov \proves \exists\beta(\beta \lrf \forall \xi (B(\xi) \lthen \exists \zeta A(\xi,\zeta))).\]
  Work in \(\el + \wkl + \markov\) and assume \(\beta \lrf \forall \xi(B(\xi) \lthen \exists\zeta A(\xi,\zeta)).\)
  Unpacking the definition of \(\lrf,\) we see that if \(\gamma \lrf B(\xi)\) then \(\beta|\xi|\gamma\converges,\) \(\vset{\beta|\xi|\gamma} \neq \varnothing,\) and \[\forall \zeta \in \vset{\beta|\xi|\gamma} (\snd{\zeta} \lrf A(\xi,\fst{\zeta})).\]
  Since \(B(\xi) \in \lneg,\) it follows from Lemma~\ref{L:LNegative} that there is a term \(\omega_B\) such that \[B(\xi) \liff \omega_B|\xi\converges \land \omega_B|\xi \lrf B(\xi).\]
  Finally, since \(A(\xi,\zeta) \in \lcon,\) it follows from Lemma~\ref{L:LConservation} that \[\alpha = \Lambda\xi.\iota|\pr{\Lambda\zeta.\fst{\zeta}}{\beta|\xi|(\omega_B|\xi)}\] is as required, where \(\iota\) is as in Lemma~\ref{L:Image}.
\end{proof}

\noindent
As for Proposition~\ref{P:KExtract}, we can add to the theories in Proposition~\ref{P:LExtract} any collection of sentences from \(\lneg.\)

\begin{definition}
  Let \(\lmax\) be the collection of all sentences \(A\) from \(\lneg\) such that \(\rca + \wkl \proves A.\)
  In other words, \(\lmax\) consists of all consequences of the law of excluded middle which belong to the syntactic class \(\lneg.\)
\end{definition}

\noindent
Note that \(\lmax\) includes the Markov principle \(\markov\) as well as the lesser limited principle of omniscience \(\llpo\) (see Section~\ref{S:Applications}).

\begin{corollary}\label{C:LUniformization}
  Suppose \(B(\xi)\) is from \(\lneg\) and \(A(\xi,\zeta)\) is from \(\lcon.\)
  If \[\el + \wkl + \lgc + \lmax \proves \forall\xi(B(\xi) \lthen \exists\zeta A(\xi,\zeta))\] then \[\rca + \wkl \proves \forall\xi(\forall n B(\xi_n) \lthen \exists\zeta \forall n A(\xi_n,\zeta_n)).\]
\end{corollary}

\begin{proof}
  Suppose that \[\el + \wkl  + \lgc + \lmax \proves \forall\xi(B(\xi) \lthen \exists\zeta A(\xi,\zeta)).\]
  By Proposition~\ref{P:LExtract}, we know that \[\el + \wkl + \lmax \proves \exists\alpha\forall\xi(B(\xi) \lthen \alpha|\xi\converges \land \vset{\alpha|\xi} \neq \varnothing \land \forall\zeta \in \vset{\alpha|\xi} A(\xi,\zeta)).\]
  Now work in \(\rca + \wkl,\) which extends \(\el + \wkl + \lmax.\)
  Find \(\alpha\) such that if \(B(\xi)\) then \[\alpha|\xi\converges \land \vset{\alpha|\xi} \neq \varnothing \land \forall\zeta \in \vset{\alpha|\xi} A(\xi,\zeta).\]
  If \(\xi\) is such that \(\forall n B(\xi_n),\) then \[\forall n(\alpha|\xi_n\converges \land \vset{\alpha|\xi_n} \neq \varnothing).\]
  By \cite[Lemma~VIII.2.4]{Simpson09}, we can find a \(\zeta\) such that \(\zeta_n \in \vset{\alpha|\xi_n}\) for every \(n.\)
  It then follows that \(\forall n A(\xi_n,\zeta_n).\)
  We have just shown that \[\rca + \wkl \proves \forall\xi(\forall n B(\xi_n) \lthen \exists\zeta \forall n A(\xi_n,\zeta_n)).\qedhere\]
\end{proof}

\begin{remark}\label{R:LInduction}
  As in Remark~\ref{R:KInduction}, it would be desirable to eliminate the induction assumptions from Corollary~\ref{C:LUniformization}.
  Unfortuantely, van Oosten's arguments from~\cite{VanOosten90} do appear to make some use of this inductive assumption.
  Close inspection reveals that these uses are limited to \(\Pi^0_1\)-bounding, therefore Corollary~\ref{C:LUniformization} does have an analogue with \(\rca + \wkl\) replaced by \(\rca_0 + \wkl +  \mathsf{B}\Pi^0_1.\)
\end{remark}

\section{Applications}\label{S:Applications}

To compare the earlier results of Hirst and Mummert with ours, it is useful to compare the syntactic restrictions involved, specifically~\cite[Theorem~3.6]{HirstMummert11} since the syntactic conditions for~\cite[Theorem~5.6]{HirstMummert11} are even more restrictive.

The analogue of \(\kneg\) and \(\lneg\) for Hirst and Mummert are \(\exists\)-free formulas: formulas built in the usual manner but without the use existential quantifiers nor disjunctions.
The \(\exists\)-free fromulas are a proper subset of \(\kneg\) and hence \(\lneg\) since some existential quantifiers are allowed by the first clause of Definition~\ref{D:KNeg}, and still more are allowed by the second clause of Definition~\ref{D:LNeg}.

The analogue of \(\kcon\) and \(\lcon\) for Hirst and Mummert is the class \(\Gamma_1,\) which is defined in exactly the same way except that hypotheses of conditionals are restricted to \(\exists\)-free formulas.
Thus, \(\Gamma_1\) is also a proper subset of \(\kcon\) and hence \(\lcon.\)

To illustrate the difference, consider the familiar statement:
\begin{quote}
  \textit{Every \(n \times n\) matrix with nonzero determinant has an inverse.}
\end{quote}
The ``nonzero determinant'' hypothesis is not expressible by an \(\exists\)-free formula since to say that a Cauchy real or complex number (see below) is apart from zero requires an existential quantifier.
However, this hypothesis is expressible in \(\kneg.\)
Since \el\ proves that every \(n \times n\) matrix with nonzero determinant has an inverse, it follows that the sequential form of the above statement is provable in \rca.
The reader should not feel too enlightened by this simple example since the obvious proof is nothing more than Cramer's rule.

On the other hand, the results of Hirst and Mummert allow for higher types, while ours only involve first-order and second-order types.
Therefore, there is a vast sea of statements for which the results of Hirst and Mummert apply but ours do not.
Still, the non-provability examples that Hirst and Mummert give are all second-order, so they all have equivalents in our context.
In particular, neither \(\el + \gc + \kmax\) nor \(\el + \wkl + \lgc + \lmax\) prove that every \(n \times n\) matrix has a Jordan canonical form.

\subsection{Trichotomy and dichotomy for Cauchy reals}
 
A \emph{Cauchy real} is a rational valued function \(\alpha\) such that \(|\alpha(s) - \alpha(t)| \leq 2^{-s}\) for all \(s < t.\)
We write \(\alpha \in \R^C\) to abbreviate the statement that \(\alpha\) is a Cauchy real.
If \(\alpha, \beta \in \R^C\) then we define \[\alpha = \beta \liff \forall s (|\alpha(s) - \beta(s)| \leq 2^{1-s}).\]
We also define \[\alpha > \beta \liff \exists s (\alpha(s) - \beta(s) > 2^{1-s})\] and \[\alpha \leq \beta \liff \forall s (\alpha(s) - \beta(s) \leq 2^{1-s}).\]
Note that \(\alpha \leq \beta \liff \lnot(\alpha > \beta)\) and \(\alpha > \beta \lthen \lnot(\alpha \leq \beta),\) but the implication \(\lnot(\alpha \leq \beta) \lthen \alpha > \beta\) is equivalent to the Markov principle \markov.

The \emph{trichotomy law} \[\alpha < \beta \lor \alpha = \beta \lor \alpha > \beta\] and the formally weaker \emph{dichotomy law} \[\alpha \leq \beta \lor \alpha \geq \beta\] are both consequences of the law of excluded middle.
However, over \(\el_0\) these are respectively equivalent to the \emph{limited principle of omniscience}
\begin{axiom}{\lpo}
  \exists n (\xi(n) \neq 0) \lor \forall n (\xi(n) = 0)
\end{axiom}
and the \emph{lesser limited principle of omniscience}
\begin{axiom}{\llpo}
  \lnot(\exists n(\xi(n) \neq 0) \land \exists n(\zeta(n) \neq 0)) \lthen 
  \forall n (\xi(n) = 0) \lor \forall n (\zeta(n) = 0)
\end{axiom}
(see \cite{DoraisHirstShafer} for details).

\begin{proposition}\label{P:LLPO}
  The following equivalent statements are both provable in \(\rca_0,\) but neither is provable in \(\el + \gc + \kmax.\)
  \begin{enumerate}[\upshape\bfseries(a)]
  \item The dichotomy law for Cauchy reals.
  \item The lesser limited principle of omniscience.
  \end{enumerate}
\end{proposition}

\begin{proof}
  The dichotomy law can be stated as \[\forall\alpha,\beta(\alpha,\beta \in \R^C \lthen \exists y((y = 0 \lthen \alpha \leq \beta) \land (y \neq 0 \lthen \alpha \geq \beta))).\]
  Inspection shows that this is in the form required for Corollary~\ref{C:KUniformization}.
  Dorais, Hirst, and Shafer~\cite{DoraisHirstShafer} have shown that the corresponding sequential form is equivalent to \(\wkl\) over \(\rca_0,\) it follows that the dichotomy law is not provable in \(\el + \gc + \kmax.\)
\end{proof}

\begin{proposition}\label{P:LPO}
  The following equivalent statements are both provable in \(\rca,\) but neither is provable in \(\el + \wkl + \lgc + \lmax.\)
  \begin{enumerate}[\upshape\bfseries(a)]
  \item The trichotomy law for Cauchy reals. 
  \item The limited principle of omniscience. 
  \end{enumerate}
\end{proposition}

\begin{proof}
  The trichotomy law can be stated as: for all \(\alpha, \beta \in \R^C,\) \[\exists y((y = 0 \lthen \alpha < \beta) \land (y = 1 \lthen \alpha > \beta) \land (y > 1 \lthen \alpha = \beta)).\]
  This statement is in the form required for Corollary~\ref{C:LUniformization}.
  Dorais, Hirst, and Shafer~\cite{DoraisHirstShafer} have shown that the corresponding sequential form is equivalent to arithmetic comprehension over \(\rca_0,\) it follows that the trichotomy law is not provable in \(\el + \wkl + \lgc + \lmax.\)
\end{proof}

\subsection{Dedekind reals and Cauchy reals}

A \emph{Dedekind real} is a decidable set \(\delta\) of rationals such that \[\exists p, q \in \Q (p \in \delta \land q \notin \delta) \land \forall p, q \in \Q (p \in \delta \land q \notin \delta \lthen p < q).\]
We write \(\delta \in \R^D\) to abbreviate the fact that \(\delta\) is a Dedekind real.
We say that a Cauchy real \(\alpha\) and a Dedekind real \(\delta\) are equivalent when \[\forall s \in \N \forall p, q \in \Q (p \in \delta \land q \notin \delta \lthen \lnot(\alpha(s) + 2^{1-s} < p \lor q < \alpha(s) - 2^{1-s})).\]

\begin{proposition}\mbox{}
  \begin{enumerate}[\upshape\bfseries(a)]
  \item \(\el_0\) proves that every Dedekind real has an equivalent Cauchy real.
  \item \(\el_0 + \wkl\) proves that every Cauchy real has an equivalent Dedekind real.
  \end{enumerate}
\end{proposition}

\begin{proof}
  The proof of part~(a) is straightforward, so we only prove part~(b).

  Suppose that \(\alpha\) is a Cauchy real.
  Fix an enumeration \(\seq{q_i}_{i=0}^\infty\) of \(\Q.\)
  Let \(R(\alpha,x)\) denote the statement 
  \begin{multline*}
    \forall i < |x|((x(i) = 0 \lor x(i) = 1) \land\\ 
    (\exists s \leq |x|(q_i < \alpha(s) - 2^{1-s}) \lthen x(i) = 1) \land\\ 
    (\exists s \leq |x| (q_i > \alpha(s) + 2^{1-s}) \lthen x(i) = 0).
  \end{multline*}
  Then the decidable set \(\delta = \set{q_i : \xi(i) = 1}\) is a Dedekind real equivalent to \(\alpha\) if and only if \(\forall n R(\alpha,\res{\xi}{n}).\)
  Since \(\forall n \exists x (|x| = n \land R(\alpha,x))\) it follows from \wkl\ that there is a Dedekind real \(\delta\) which is equivalent to \(\alpha.\)
\end{proof}

\noindent
Of course, \(\rca_0\) proves that every Cauchy real has an equivalent Dedekind real.
However, the usual proof of this fact is non uniform since it relies on first deciding whether or not the Cauchy real represents a rational number.
Such lack of uniformity is actually necessary as the next proposition shows.

\begin{proposition}\label{P:CauchyDedekind}
  The system \(\el + \gc + \kmax\) does not prove that every Cauchy real has an equivalent Dedekind real.
\end{proposition}

\begin{proof}
  Formally, the statement that every Cauchy real has an equivalent Dedekind real is: for every \(\alpha \in \R^C\) there is a \(\delta \in \R^D\) such that \[\forall p,q \in \Q \forall s \in \N (p \in \delta \land q \notin \delta \lthen \lnot(\alpha(s) + 2^{1-s} < p \lor q < \alpha(s) - 2^{1-s})).\]
  Inspection shows that this has the right form for Corollary~\ref{C:KUniformization}.
  However, Hirst~\cite{Hirst07} has shown that the sequential form of this statement is equivalent to \wkl\ over \rcao.
  It follows that the statement is not provable in \(\el + \gc + \kmax.\)
\end{proof}

\noindent
Note that dichotomy is trivially true for Dedekind reals (simply check in which half \(0\) is).
Thus, Proposition~\ref{P:CauchyDedekind} is actually a corollary of Proposition~\ref{P:LLPO}.

\subsection{The fundamental theorem of algebra}

Cauchy complex numbers are pairs \(\seq{\xi_0,\xi_1}\) where \(\xi_0, \xi_1 \in \R^C.\)
These are intended to represent the real and imaginary parts of the complex number.
Thus, we write \(\xi \in \C^C\) to abbreviate \(\fst{\xi}, \snd{\xi} \in \R^C.\)
Addition and multiplication on complex numbers are defined as usual; it is not difficult to check that \(\el_0\) proves that \(\C^C\) is a field.
However, \(\el + \gc\) does not prove that \(\C^C\) is algebraically complete.

\begin{proposition}\label{P:SquareRoot}
  \(\el + \gc + \kmax\) does not prove that every complex number has a square root.
\end{proposition}

\begin{proof}
  Suppose on the contrary that \(\el + \gc\) does prove that every complex number has a square root. 
  It follows that from Proposition~\ref{P:KExtract} that \(\el\) proves the existence of some \(\alpha\) such that \[(\forall \xi)(\xi \in \C^C \lthen \alpha|\xi\converges \land \alpha|\xi \in \C^C \land (\alpha|\xi)^2 = \xi).\]
  Since this statement is in \(\kcon,\) it follows that \(\el\) proves the existence of such an \(\alpha.\)
  This is impossible since the axioms of \(\el\) are classically valid and there is no total continuous function on the complex numbers that selects one of the two square roots of its argument.
\end{proof}

\noindent
The use of Proposition~\ref{P:KExtract} instead of Corollary~\ref{C:KUniformization} was necessary for this argument since \(\rca_0\) does prove the sequential form \[\forall\xi(\forall n(\xi_n \in \C^C) \lthen \exists\zeta\forall n(\zeta_n \in \C^C \land \xi_n = \zeta_n^2)).\]
In particular, the converse of Corollary~\ref{C:KUniformization} is false.

While the fundamental theorem of algebra is not provable in \(\el + \gc + \kmax,\) it is provable in \(\el + \wkl.\)

\begin{proposition}
  \(\el + \wkl\) proves that
  \[\forall\xi_1,\dots,\xi_n \in \C^C \exists\zeta \in \C^C(\zeta^n + \xi_1\zeta^{n-1} + \cdots + \xi_n = 0).\]
\end{proposition}

\noindent
This is because  \(\el\) proves that for any coefficients \(\xi_1,\dots,\xi_n \in \C^C,\) there is a function \(\alpha\) such that \[\vset{\alpha} = \set{\zeta \in \C^C : \zeta^n + \xi_1\zeta^{n-1} + \cdots + \xi_n = 0}.\]
Then, by proving the existence of approximate roots, \(\el + \wkl\) proves that \(\vset{\alpha} \neq \varnothing.\)

\bibliographystyle{amsplain}
\bibliography{gc}

\end{document}